\documentclass[a4paper,11pt]{article}

\usepackage{amsbsy}
\usepackage{amscd}
\usepackage{amsfonts}
\usepackage{amsmath}
\usepackage{amssymb}
\usepackage{amsthm}
\usepackage[english]{babel}
\usepackage{bookmark}
\usepackage{color}
\usepackage{eso-pic}
\usepackage{float}
\usepackage[letterpaper,left=1.25in,right=1.25in,top=1.25in,bottom=1.25in]{geometry}
\usepackage{graphicx}
\usepackage{inputenc}
\usepackage{pinlabel}
\usepackage{scalerel,stackengine}
\usepackage{subfigure}
\usepackage{tikz}
\usepackage{xfrac}
\usepackage[all]{xy}

\usetikzlibrary{matrix,arrows}

\newtheorem{lemma}{Lemma}[section]
\newtheorem{teo}[lemma]{Theorem}
\newtheorem{prop}[lemma]{Proposition}

\theoremstyle{definition}
\newtheorem{defn}[lemma]{Definition}

\theoremstyle{remark}
\newtheorem{rem}[lemma]{Remark}

\newcommand{\nota} [1] {\caption{\footnotesize{#1}}}

\newcommand{\matR} {\ensuremath {\mathbb{R}}}

\newcommand{\matZ} {\ensuremath {\mathbb{Z}}}

\newcommand{\matH} {\ensuremath {\mathbb{H}}}
\newcommand{\matS} {\ensuremath {\mathbb{S}}}

\newcommand{\calX} {\ensuremath {\mathcal{X}}}
\newcommand{\calR} {\ensuremath {\mathcal{R}}}
\newcommand{\calN} {\ensuremath {\mathcal{N}}}

\newcommand{\calL} {\ensuremath {\mathcal{L}}}
\newcommand{\calS} {\ensuremath {\mathcal{S}}}
\newcommand{\calM} {\ensuremath {\mathcal{M}}}
\newcommand{\calC} {\ensuremath {\mathcal{C}}}

\newcommand{\calD} {\ensuremath {\mathcal{D}}}
\newcommand{\calE} {\ensuremath {\mathcal{E}}}
\newcommand{\calF} {\ensuremath {\mathcal{F}}}
\newcommand{\calP} {\ensuremath {\mathcal{P}}}
\newcommand{\calT} {\ensuremath {\mathcal{T}}}

\newcommand{\calH}{\ensuremath {\mathcal{H}}}

\newcommand{\calG}{\ensuremath {\mathcal{G}}}
\newcommand{\calZ}{\ensuremath {\mathcal{Z}}}

\stackMath
\newcommand\reallywidehat[1]{%
\savestack{\tmpbox}{\stretchto{%
  \scaleto{%
    \scalerel*[\widthof{\ensuremath{#1}}]{\kern-.6pt\bigwedge\kern-.6pt}%
    {\rule[-\textheight/2]{1ex}{\textheight}}
  }{\textheight}%
}{0.5ex}}%
\stackon[1pt]{#1}{\tmpbox}%
}
\begin{document}
\title{New hyperbolic 4-manifolds of low volume}
\author{Stefano Riolo, Leone Slavich}
\date{}
\maketitle

\begin{abstract}
\noindent 
We prove that there are at least two commensurability classes of (cusped, arithmetic) minimal-volume hyperbolic $4$-manifolds.
Moreover, by applying a well-known technique due to Gromov and Piatetski-Shapiro, we build the smallest known non-arithmetic hyperbolic $4$-manifold.\end{abstract}

\section{Introduction}
A \emph{hyperbolic manifold} is a manifold equipped with a Riemannian metric of constant sectional curvature equal to $-1$.
Throughout this paper, hyperbolic manifolds are assumed to be complete and of finite volume.

An important invariant of a hyperbolic manifold is its volume.
Since it is often regarded as a measure of complexity, it is reasonable to look for manifolds of low volume.
In this regard, recall that the Gauss-Bonnet formula relates the volume of a hyperbolic $4$-manifold $M$ to its Euler characteristic in the following way:
$$\mathrm{Vol}(M)=\frac{4\pi^2}{3}\chi(M).$$
Moreover, by the work of Wang \cite{Wang}, for any $n\geq4$ and $V>0$, there is at most a finite number of (isometry classes of) hyperbolic $n$-manifolds with volume bounded by $V$.

Let us now draw attention to \emph{minimal-volume} hyperbolic manifolds.
In dimension two, there are uncountably many such manifolds. Up to diffeomorphism, by the Gauss-Bonnet formula these are just three: the connected sum of three projective planes, the thrice-punctured sphere and the once-punctured torus. The former is closed, while the latter two are cusped.
In contrast, in dimension three there is a unique orientable hyperbolic manifold of minimal volume: the so called Fomenko-Matveev-Weeks manifold \cite{FM,Weeks,weeks_mfd}, which is closed. The smallest cusped hyperbolic 3-manifold is the Gieseking manifold \cite{gies_mfd} (this manifold will play a role in this paper).

From now on, let us focus on dimension four. 
In $2000$, Ratcliffe and Tschantz produced a census of $1171$ cusped hyperbolic $4$-manifolds, all tessellated by a single copy of a hyperbolic regular polytope, the ideal right-angled $24$-cell $\calC$, which has volume
$$V_{\mathrm{min}}=\frac{4\pi^2}{3}.$$
These manifolds are thus of minimal volume.
For comparison, the smallest known closed hyperbolic $4$-manifolds have volume $8\cdot V_{\mathrm{min}}$ \cite{CM,Long}, but the minimal volume of a closed hyperbolic $4$-manifold is still unknown. 

At the moment, computing the exact number of hyperbolic $4$-manifolds with volume $V_{\mathrm{min}}$ seems to be an unrealistic expectation. An explicit bound is still unknown, and this number may be enormous. Counting such manifolds up to commensurability is perhaps a simpler task. (Recall that two manifolds are \emph{commensurable} if there is a third manifold finitely covering both.)  
For instance, all the manifolds from the Ratcliffe and Tschantz's census are commensurable.

In a recent survey about hyperbolic 4-manifolds \cite[Section 4, Question 6]{BruSurv}, Martelli asks whether all hyperbolic $4$-manifolds of volume $V_{\mathrm{min}}$ are commensurable.
The main result of the present paper is the following:
\begin{teo}\label{cor:cor1}
There exist at least two commensurability classes of (cusped, arithmetic) hyperbolic $4$-manifolds containing an orientable manifold of minimal volume.  
\end{teo}

Each of these two commensurability classes is associated to a (non-compact, arithmetic) hyperbolic Coxeter polytope: the ideal regular $24$-cell $\calC$, and the ideal rectified $5$-cell $\calR$, respectively.
These two commensurability classes are represented by Coxeter diagrams in Figure \ref{fig:diagrammi} (left).
The manifolds of Ratcliffe and Tschantz are commensurable (in the orbifold sense, see Section \ref{sec:commensurabilty-invariants}) with the $24$-cell $\calC$.
Another minimal-volume manifold, commensurable with the rectified $5$-cell $\calR$, is obtained by slightly modifying a construction of the second author \cite{S}. 
The reflection groups associated to these two Coxeter polytopes are arithmetic. This allows us to distinguish their commensurability classes by applying the work of Maclachlan \cite{Mac}.

Having established the existence of two commensurability classes of minimal volume hyperbolic $4$-manifold, we turn our attention to manifolds with twice the minimal volume, namely $8\pi^2/3$.
Our main objective is constructing a (non-orientable) hyperbolic $4$-manifold $\calN$ of volume $2\cdot V_{\mathrm{min}}$, commensurable with a Coxeter polytope $\calP$ first introduced by Kerckhoff and Storm in \cite{KS} and further studied in detail in \cite{riomar} (see our Theorem \ref{teo:main1}).
The manifold $\calN$ is explicitly built by pairing the facets of two copies of the polytope $\calP$.
Our construction is similar to that of another manifold of volume $2\cdot V_{\mathrm{min}}$ built in \cite{riomar}. In fact, these two manifolds are commensurable, although not isometric. Thus, there is no reason to consider our manifold $\calN$ special. On the other hand, its construction suggests that the polytope $\calP$ may be used to build more manifolds of low volume.

Finally, we note that all the known examples of minimal-volume hyperbolic $4$-manifolds are arithmetic, so one could wonder whether there exists a non-arithmetic hyperbolic $4$-manifold of volume $V_{\mathrm{min}}$ or, more generally, what the minimal volume of a non-arithmetic hyperbolic $4$-manifold might be. We prove the following:

\begin{teo}\label{teo:main3}
There exist two non-arithmetic cusped hyperbolic 4-manifolds $\calH$ and $\calH'$ such that
\begin{itemize}
\item $\calH$ is non-orientable and $\chi(\calH)=3$,
\item $\calH'$ is orientable and $\chi(\calH')=5$.
\end{itemize} 
\end{teo}
The proof is a simple application of the well-known ``interbreeding'' technique introduced by Gromov and Piatetski-Shapiro \cite{G-PS}.
The manifold constructed in \cite{S} has totally geodesic boundary isometric to the figure-eight knot complement. By cutting our manifold $\calN$ (or its orientable double covering) along a totally geodesic hypersurface, we get another hyperbolic manifold bounding the figure-eight knot complement. 
We glue the two manifolds through an isometry of their boundaries to get the non-arithmetic $\calH$ and $\calH'$.
To the best of the authors' knowledge, these are the smallest known examples of non-arithmetic hyperbolic $4$-manifolds.

The paper is organized as follows: in Section \ref{sec:polytopes} we introduce the Coxeter polytopes $\calR$ and $\calP$, describe their combinatorial and geometric properties, and give their commensurability invariants. Theorem \ref{cor:cor1} is proved in Section \ref{sec:commensurabilty-invariants}.
In Section \ref{sec:minvolume}, we build the manifold $\calN$ commensurable with the Kerchoff-Storm polytope and prove Theorem \ref{teo:main1}. The proof of Theorem \ref{teo:main3} follows in Section \ref{sec:nonarithemtic}.

\begin{figure}\label{fig:diagrammi}
\begin{center}
\includegraphics[width=0.75\textwidth]{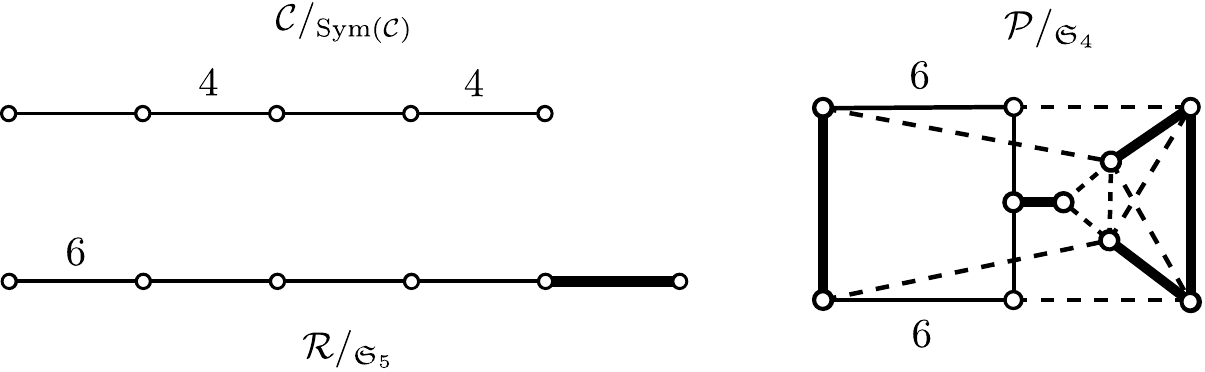}
\end{center}
\nota{The Coxeter diagrams of three polytopes obtained from $\calC$, $\calR$ and $\calP$ as quotients by some groups of symmetries.
The polytope $\calC/_{\mathrm{Sym}(\calC)}$ (top left) is the characteristic simplex of the 24-cell $\calC$; for the remaining two polytopes see Lemma \ref{lemma:pyramid} and the end of Section \ref{sec:polytope}.
The notation for Coxeter diagrams is explained in Section \ref{sec:rectified-5-cell}.}
\end{figure}

\paragraph{Acknowledgments.}
The authors are grateful to Ruth Kellerhals, Sasha Kolpakov and Bruno Martelli for some crucial observations. In particular, we thank Ruth Kellerhals and Sasha Kolpakov for suggesting a simpler proof of Lemma \ref{lemma:pyramid}.
Figures \ref{fig:facets} and \ref{fig:triangulation_figure8} come from the paper \cite{riomar} and were originally drawn by Bruno Martelli (actually, Figure \ref{fig:triangulation_figure8} is taken from his beautiful book \cite{BruBook}).
We also thank an anonymous referee for pointing out a major mistake in a previous version of the manuscript.

The first author was supported by the research fellowship ``Deformazioni di strutture iperboliche in dimensione quattro'', by the Mathematics Department of the University of Pisa.
The second author was supported by a grant from ``Scuola di Scienze di base Galileo Galilei'', and wishes to thank the Department of Mathematics of the University of Pisa for the hospitality while this work was conceived and written.

\section{Two Coxeter polytopes}\label{sec:polytopes}

In this Section, we first introduce the Coxeter polytopes $\calP$ and $\calR$, and then show that they are not commensurable. Finally, we prove Theorem \ref{cor:cor1}.

\subsection{The rectified $5$-cell}\label{sec:rectified-5-cell}
Here, we briefly introduce the ideal hyperbolic rectified 5-cell $\calR$. We refer the reader to \cite{KS2014,S} for more details.

\begin{defn}\label{def:5cell}
Consider a regular Euclidean $4$-simplex $\Delta\subset\matR^4$, normalised such that the midpoints of its edges $P_1,\ldots,P_{10}$ belong to the unit sphere $\matS^3$. Interpret now the unit ball $\mathbb B^4$ as the hyperbolic space $\matH^4$ in the Klein-Beltrami model. The \emph{rectified 5-cell} $\calR$ is the convex hull in $\matH^4$ of the ideal points $P_1,\ldots,P_{10}\in\partial_\infty\matH^4$.
(Of course, we have defined the polytope $\calR$ up to isometries of $\matH^4$.)
\end{defn}

The polytope $\calR$ has ten facets: five ideal regular octahedra and five ideal regular tetrahedra.

\begin{figure}[htbp]\label{fig:cuspR}
\begin{center}
\includegraphics[width=0.09\textwidth]{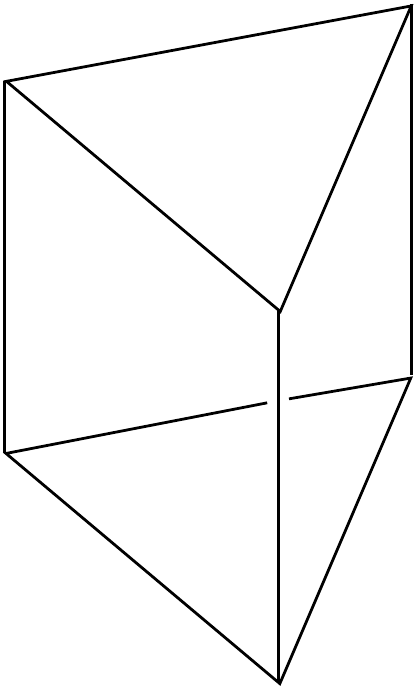}
\nota{The vertex figure of the rectified $5$-cell $\calR$ is a Euclidean right prism over an equilateral triangle.  All of its edges have equal length.}
\end{center}
\end{figure}

We note that the symmetry group of $\calR$ acts transitively on the set of its vertices. 
The vertex figure $L$ of $\calR$ is a right Euclidean prism over an equilateral triangle, with all edges of equal length. At each vertex, there are three octahedra meeting side-by-side, corresponding to the square faces, and two tetrahedra, corresponding to the triangular faces.

The dihedral angle between two octahedral facets is therefore equal to $\pi/3$, while the dihedral angle between a tetrahedral and an octahedral facet is equal to $\pi/2$. 
An important consequence of this fact is that $\calR$ is a Coxeter polytope.

As shown in \cite{KS2014}, the volume of the rectified $5$-cell is
\begin{equation}\label{eq:volume}
\mathrm{Vol}(\calR)=\frac{2\pi^2}{9}=\frac16V_{\mathrm{min}}.
\end{equation}

The symmetry group $\mathrm{Sym}(\calR)$ of the polytope $\mathcal{R}$ is clearly isomorphic to $\mathrm{Sym}(\Delta)$, which is isomorphic to the symmetric group $\mathfrak{S}_5$.
Moreover, we obtain the following one-to-one correpondences (see also \cite{KS2014}):
\begin{enumerate}
\item $\{\mbox{Vertices of }\calR\}\leftrightarrow\{\mbox{Edges of }\Delta\}$
\item $\{\mbox{Tetrahedral facets of }\calR\}\leftrightarrow\{\mbox{Vertices of }\Delta\}$
\item $\{\mbox{Octahedral facets of }\calR\}\leftrightarrow\{\mbox{Facets of }\Delta\}$
\end{enumerate}

The next Lemma will be used in Section \ref{sec:commensurabilty-invariants}.
By a \emph{hyperbolic $n$-pyramid}, we mean the convex hull in $\matH^n$ of a (possibly ideal) point
and an $(n-1)$-dimensional polytope which is not a simplex -- see \cite{tumark}.

We will adopt the following notation for (generalised) Coxeter diagrams -- see \cite{Vin}. The thick edges correspond to facets which are tangent at infinity.  The dotted edges
correspond to facets which are at positive distance.  The thin unlabeled edges correspond to facets which intersect at angle $\pi/3$. Vertices not joined by an edge correspond to facets intersecting orthogonally.  Finally, edges labeled by $\nu$ (not necessarily integer) correspond to facets intersecting at an angle of $\pi/\nu$.

\begin{lemma}\label{lemma:pyramid}
The quotient $\calR/_{\mathfrak S_5}$ of the rectified $5$-cell under the action of its symmetry group is isometric to the hyperbolic Coxeter pyramid whose Coxeter diagram is:
\begin{equation}\label{diagr:pyramid}
\includegraphics[width=0.4\textwidth]{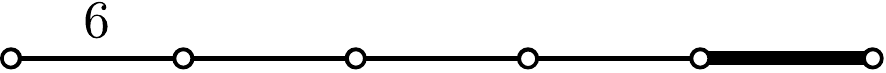}
\end{equation}
\end{lemma}
\begin{proof}
%
For simplicity, let us first consider a regular $4$-simplex $\Delta_{\alpha}\subset\matH^4$ with dihedral angles $\alpha$ (where $\alpha>\arccos\frac13$) and its symmetry group $\mathfrak S_5$.
It is well known that the quotient $\Delta_{\alpha}/_{\mathfrak S_5}$ is an orthoscheme with generalised Coxeter diagram

\begin{center}
\includegraphics[width=0.27\textwidth]{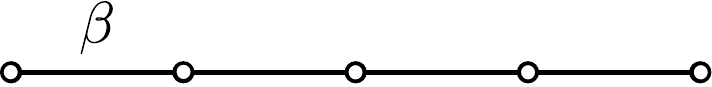}
\end{center}
where $\beta = 2\pi/\alpha$.

The simplex $\Delta_{\alpha}$ is ideal for $\alpha=\arccos\frac13$.
If instead $\alpha\in\big(\frac\pi3,\arccos\frac13\big)$, the simplex $\Delta_{\alpha}$ is hyper-ideal, and
each hyper-ideal vertex determines a dual hyperplane in $\matH^n$. We can thus consider the corresponding truncated simplex $\mathrm{tr}(\Delta_{\alpha})$ (whose symmetry group is always $\mathfrak{S}_5$).
In that case, the quotient $\mathrm{tr}(\Delta_{\alpha})/_{\mathfrak S_5}$ has generalised Coxeter diagram
\begin{center}
\includegraphics[width=0.34\textwidth]{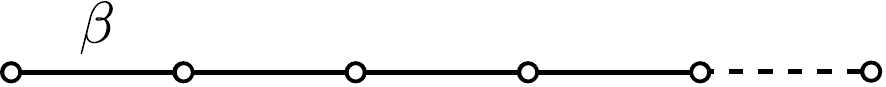}
\end{center}
where the additional node corresponds to the truncation.
In other words, the polytope $\mathrm{tr}(\Delta_{\alpha})/_{\mathfrak S_5}$ can be obtained by truncating $\Delta_{\alpha}/_{\mathfrak S_5}$ with the hyperplane dual to its hyper-ideal vertex.

As the angle $\alpha$ approaches $\pi/3$, the truncation $\mathrm{tr}(\Delta_{\alpha})$ tends to the rectification $\calR$, and the generalised Coxeter diagram of the quotient $\calR/_{\mathfrak S_5}$ is given by \eqref{diagr:pyramid}.
\end{proof}

Conversly, we see that the rectified $5$-cell $\calR$ can be obtained from the pyramid represented by \eqref{diagr:pyramid} as its orbit by the action of the group $\mathfrak S_5$ generated by reflections in the hyperplanes
associated to the sub-diagram
\begin{center}
\includegraphics[width=0.225\textwidth]{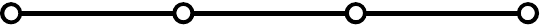}
\end{center}
This subdiagram is the link of the vertex of $\calR/_{\mathfrak S_5}$ corresponding to the barycenter of the original simplex $\Delta=\Delta_{\frac\pi3}$.
Such operation is also known as the ``Wythoff construction'' -- see \cite{coxeter}.

\subsection{The Kerckhoff-Storm polytope}\label{sec:polytope}
In this section, we describe the hyperbolic Coxeter polytope $\calP$ that tessellates the manifold $\calN$ of Theorem \ref{teo:main1}.
The reflection lattice associated to $\calP$ was recently discovered by Kerckhoff and Storm in \cite[Section 13.2]{KS}.
The geometry and combinatorics of the polytope $\calP$ are carefully described in \cite[Proposition 3.14]{riomar}.
We refer to these papers for more details -- in particular, the reader can find the proofs of all the upcoming facts in \cite{riomar}.

Let us consider the hyperboloid model of the hyperbolic 4-space, that is, we set
$$\matH^4=\big\{ v\in\matR^{1,4}\ \big|\ \langle v,v\rangle=-1,\ \langle v,e_0\rangle<0\big\},$$
where $\matR^{1,4}$ is the Minkowski space with standard basis $e_0,\ldots,e_4$ and Lorentzian product $\langle\ ,\ \rangle$ given by
$$\langle e_i,e_j\rangle=
\begin{cases}
-1\mathrm{\ if\ }i=j=0,\\
\ \ 1\mathrm{\ if\ }i=j\neq0,\\
\ \ 0\mathrm{\ if\ }i\neq j.
\end{cases}$$
Every space-like vector $v\in\matR^{1,4}$ (i.e., such that $\langle v,v\rangle>0$) determines a half-space of $\matH^4$ 
$$H_v=\big\{x\in\matH^4\big|\langle x,v\rangle\leq0\big\}.$$ 

\begin{defn}\label{def:P}
Let $\calP\subset\matH^4$ be the intersection of the 24 half-spaces determined by the 24 space-like vectors in $\matR^{1,4}$ listed in Table \ref{walls:table}.
\end{defn}

\begin{table}[htbp]
\begin{eqnarray*}
\left( \sqrt{2},1,1,1,\sqrt{5/3} \right),&
\left( \sqrt{2},1,1,1,-\sqrt{3/5} \right),&
\left(1,\sqrt{2},0,0,0\right), \\
\left( \sqrt{2},1,-1,1,-\sqrt{5/3}\right),&
\left( \sqrt{2},1,-1,1,\sqrt{3/5} \right),&
\left(1,0,\sqrt{2},0,0\right),\\
\left( \sqrt{2},1,-1,-1,\sqrt{5/3}\right),&
\left( \sqrt{2},1,-1,-1,-\sqrt{3/5} \right),&
\left(1,0,0,\sqrt{2},0\right), \\
\left( \sqrt{2},1,1,-1,-\sqrt{5/3}\right),& 
\left( \sqrt{2},1,1,-1,\sqrt{3/5}\right),& 
\left(1,0,0,-\sqrt{2},0\right),\\
\left(\sqrt{2},-1,1,-1,\sqrt{5/3}\right),& 
\left(\sqrt{2},-1,1,-1,-\sqrt{3/5}\right),& 
\left(1,0,-\sqrt{2},0,0\right), \\
\left(\sqrt{2},-1,1,1,-\sqrt{5/3}\right),& 
\left( \sqrt{2},-1,1,1,\sqrt{3/5}\right),& 
\left(1,-\sqrt{2},0,0,0\right), \\
\left(\sqrt{2},-1,-1,1,\sqrt{5/3}\right),& 
\left(\sqrt{2},-1,-1,1,-\sqrt{3/5}\right),& 
\left(\sqrt5,0,0,0,-\sqrt{6}\right), \\
\left(\sqrt{2},-1,-1,-1,-\sqrt{5/3}\right),& 
\left(\sqrt{2},-1,-1,-1,\sqrt{3/5}\right),& 
\left(\sqrt5,0,0,0,\sqrt{6}\right).
\end{eqnarray*}
\caption{The half-spaces that define $\calP$ are determined by these space-like vectors.}\label{walls:table}
\end{table}

The set $\calP$ is a hyperbolic Coxeter 4-polytope, and has 24 facets (i.e. 3-faces), 100 ridges (i.e. 2-faces), 120 edges, and 44 vertices, of which 20 are ideal.
The combinatorics and geometry of $\calP$ can be recovered from Figure \ref{fig:facets}.

\paragraph{The facets.}
The symmetry group of $\calP$ acts transitively on each of the following sets of facets:
\begin{enumerate}
\item The \emph{positive} (abbreviated P) \emph{facets} are the eight facets determined by the vectors of Table \ref{walls:table} whose last coordinate is $\pm\sqrt{5/3}$.
Four of them lie in the half-space $H_{e_4}$ (resp. $H_{-e_4}$), and are called \emph{upper} (resp. \emph{lower}) \emph{positive facets}\footnote{In \cite{KS,riomar}, these are called the ``odd (resp. even) positive walls''.}.
\item The \emph{negative} (abbreviated N) \emph{facets} are the eight facets determined by the vectors of Table \ref{walls:table} whose last coordinate is $\pm\sqrt{3/5}$.
Four of them lie in the half-space $H_{e_4}$ (resp. $H_{-e_4}$), and are called \emph{upper} (resp. \emph{lower}) \emph{negative facets}\footnote{In \cite{KS,riomar}, these are called the ``even (resp. odd) negative walls''.}.
\item The \emph{equatorial} (abbreviated E) \emph{facets} are the six facets determined by the vectors of Table \ref{walls:table} whose last coordinate is $0$.
Each such facet intersects the equatorial hyperplane $\partial H_{e_4}=\{x_4=0\}\subset\matH^4$ in an ideal quadrilateral.
\item The \emph{tetrahedral} (abbreviated T) \emph{facets} are the two facets determined by the vectors of Table \ref{walls:table} whose last coordinate is $\pm\sqrt{6}$.
These are regular ideal tetrahedra, and are the only facets whose intersection with the equatorial hyperplane $\partial H_{e_4}=\{x_4=0\}\subset\matH^4$ is empty.
The facet given by $\sqrt5e_0-\sqrt6e_4$ (resp. $\sqrt5e_0+\sqrt6e_4$) is the \emph{upper} (resp. \emph{lower}) \emph{tetrahedral facet}\footnote{The eight facets of items 3 and 4 are called the ``letter walls'' in \cite{KS,riomar}.}.
\end{enumerate}

\begin{figure}
\begin{center}
\includegraphics[width=\textwidth]{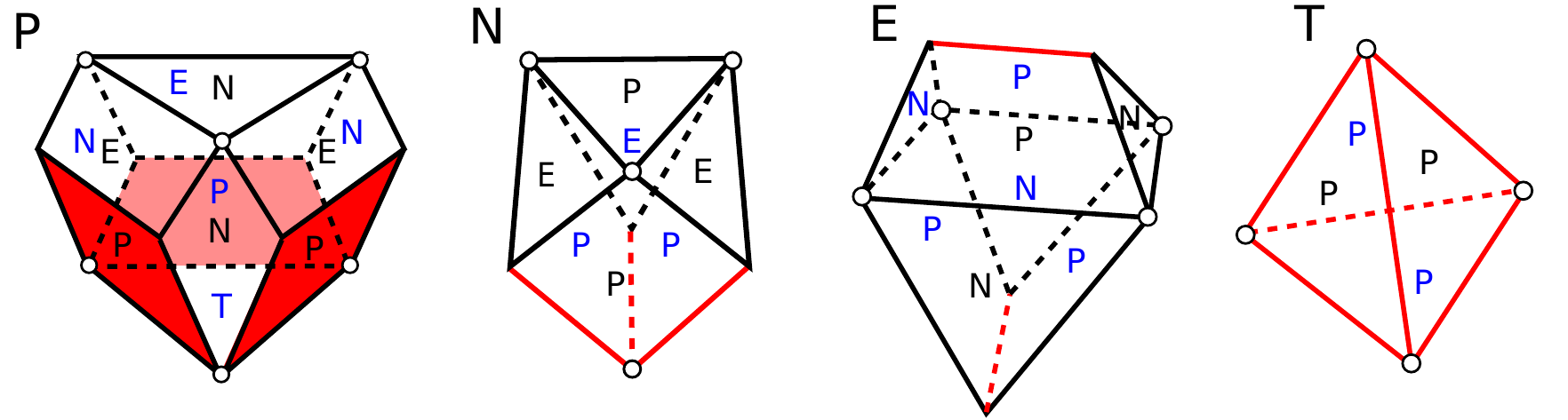}
\nota{The geometry and combinatorics of the positive, negative, equatorial and tetrahedral facets of the polytope $\calP$ (respectively labeled P, N, E and T). Ideal vertices are represented by white dots. The labels on each $2$-dimensional face show adjacencies between facets of various types. Black labels apply to faces in the foreground, and blue labels to faces in the background. In the polytope $\calP$, the white and red 2-faces have dihedral angles $\pi/2$ and $\pi/3$, respectively. Similarly, in each facet, the black edges are right-angled, while the red edges have dihedral angle $\pi/3$.}\label{fig:facets}
\end{center}
\end{figure}

We now describe some of the faces of $\calP$ of lower dimension. 
Note that, being a Coxeter polytope, $\calP$ is \emph{simple} -- meaning that every face of codimension $k$ (except the ideal vertices) is the intersection of exactly $k$ facets of $\calP$ -- see \cite{Vin}.

The faces of $\calP$ will be often distinguished by their type (rather than by isometry class), where by the \emph{type} of a face $\calF$ we mean the isometry classes of the facets of $\calP$ whose intersection is $\calF$.
For instance, an edge of $\calP$ is of type PNE if it is the intersection of a positive, a negative and an equatorial facet. However there are two distinct isometry classes of PNE edges (some edges have two ideal vertices, some others have only one), as can be seen from Figure \ref{fig:facets}. 
(For the facets of $\calP$, instead, the type coincides with the isometry class.)

\paragraph{The ridges.}
The polytope $\calP$ has dihedral angles $\pi/2$ and $\pi/3$.
The right-angled ridges are of type PN, PE, PT and NE, while the ridges with dihedral angle $\pi/3$ are of type PP, as shown in Figure \ref{fig:facets}.
We note that every ridge has some ideal vertices.

\paragraph{The vertices.}

\begin{figure}
\begin{center}
\definecolor{qqqqff}{rgb}{0,0,1}
\definecolor{ffqqqq}{rgb}{1,0,0}
\begin{tikzpicture}[line cap=round,line join=round,>=triangle 45,x=.8cm,y=.8cm]
\clip(-4.32,-2.77) rectangle (10.41,4.28);
\draw [line width=1.2pt] (-2,1)-- (-0.5,1);
\draw [line width=1.2pt] (-0.5,1)-- (-0.5,-1);
\draw [line width=1.2pt] (-0.5,-1)-- (-2,-1);
\draw [line width=1.2pt] (-2,1)-- (-4,2);
\draw [line width=1.2pt] (-4,2)-- (-2.5,2);
\draw [line width=1.2pt] (-2.5,2)-- (-0.5,1);
\draw [line width=1.2pt] (-2,-1)-- (-4,0);
\draw [line width=1.2pt] (-4,0)-- (-4,2);
\draw [line width=1.2pt,dash pattern=on 3pt off 3pt] (-4,0)-- (-2.5,0);
\draw [line width=1.2pt,dash pattern=on 3pt off 3pt] (-2.5,0)-- (-0.5,-1);
\draw [line width=1.2pt,dash pattern=on 3pt off 3pt] (-2.5,0)-- (-2.5,2);
\draw [line width=1.2pt] (-2,1)-- (-2,-1);
\draw [line width=1.2pt] (4,0.4)-- (3,-1);
\draw [line width=1.2pt] (3,-1)-- (5,-1);
\draw [line width=1.2pt] (5,-1)-- (4,0.4);
\draw [line width=1.2pt,color=ffqqqq] (4,0.4)-- (4,3.4);
\draw [line width=1.2pt] (4,3.4)-- (5,2);
\draw [line width=1.2pt,color=ffqqqq] (5,2)-- (5,-1);
\draw [line width=1.2pt] (4,3.4)-- (3,2);
\draw [line width=1.2pt,color=ffqqqq] (3,2)-- (3,-1);
\draw [line width=1.2pt,dash pattern=on 3pt off 3pt] (3,2)-- (5,2);
\draw (-1.5,0.4) node[anchor=north west] {P};
\draw [color=qqqqff](-3.58,1.53) node[anchor=north west] {P};
\draw (-3.25,0.86) node[anchor=north west] {N};
\draw [color=qqqqff](-1.9,1.1) node[anchor=north west] {N};
\draw (-2.52,1.84) node[anchor=north west] {E};
\draw [color=qqqqff](-2.81,-0.1) node[anchor=north west] {E};
\draw (4.17,1.5) node[anchor=north west] {P};
\draw (3.26,1.52) node[anchor=north west] {P};
\draw [color=qqqqff](4.0,1.0) node[anchor=north west] {P};
\draw (3.7,-0.21) node[anchor=north west] {T};
\draw [color=qqqqff](3.45,2.92) node[anchor=north west] {N};
\draw (7.84,1.1) node[anchor=north west] {P};
\draw [color=qqqqff](8.26,1.55) node[anchor=north west] {P};
\draw [color=qqqqff](8.01,0.23) node[anchor=north west] {N};
\draw (9,1) node[anchor=north west] {E};
\draw [shift={(0.88,1.52)},line width=1.2pt]  plot[domain=-0.34:0.06,variable=\t]({1*7.95*cos(\t r)+0*7.95*sin(\t r)},{0*7.95*cos(\t r)+1*7.95*sin(\t r)});
\draw [shift={(11.35,-1.94)},line width=1.2pt,color=ffqqqq]  plot[domain=2.14:2.63,variable=\t]({1*4.71*cos(\t r)+0*4.71*sin(\t r)},{0*4.71*cos(\t r)+1*4.71*sin(\t r)});
\draw [shift={(11.46,2.15)},line width=1.2pt]  plot[domain=3.55:3.96,variable=\t]({1*4.54*cos(\t r)+0*4.54*sin(\t r)},{0*4.54*cos(\t r)+1*4.54*sin(\t r)});
\draw [shift={(6.32,-0.85)},line width=1.2pt]  plot[domain=0.23:0.86,variable=\t]({1*3.78*cos(\t r)+0*3.78*sin(\t r)},{0*3.78*cos(\t r)+1*3.78*sin(\t r)});
\draw [shift={(7.52,1.81)},line width=1.2pt]  plot[domain=4.99:5.65,variable=\t]({1*3.07*cos(\t r)+0*3.07*sin(\t r)},{0*3.07*cos(\t r)+1*3.07*sin(\t r)});
\draw [shift={(8.17,-3.27)},line width=1.2pt,dash pattern=on 3pt off 3pt]  plot[domain=1.06:1.81,variable=\t]({1*3.74*cos(\t r)+0*3.74*sin(\t r)},{0*3.74*cos(\t r)+1*3.74*sin(\t r)});
\draw (-2.03,-1.63) node[anchor=north west] {(1)};
\draw (3.64,-1.68) node[anchor=north west] {(2)};
\draw (8.48,-1.7) node[anchor=north west] {(3)};
\end{tikzpicture}
\end{center}
\nota{The links of the vertices of the polytope $\calP$. Black edges are right-angled, while red edges have dihedral angle $\pi/3$. The faces of these three polyhedra are labeled (front faces in black, back faces in blue) with a symbol -- P for positive, N for negative, E for equatorial, T for tetrahedral -- denoting the isometry class of the corresponding facet of $\calP$. The Euclidean parallelepiped (1) is the link of an equatorial ideal vertex, the Euclidean prism (2) is the link of an upper or lower ideal vertex, the spherical tetrahedron (3) is the link of a finite vertex.}\label{fig:links}
\end{figure}

The link $L_v$ of each vertex $v$ of $\calP$ is a three-dimensional polyhedron, which is Euclidean (and defined up to rescaling) if $v$ is ideal, or spherical if $v$ is finite.
Each face of $L_v$ is the link of $v$, seen as vertex of an appropriate facet $\calF$ of $\calP$. 
We label each face of $L_v$ by the type of the corresponding $\calF$, as in Figure \ref{fig:links}.

The symmetry group of $\calP$ acts transitively on each of the following sets of vertices:
\begin{enumerate}
\item 12 \emph{equatorial ideal vertices}, lying in $\partial_\infty\{x_4=0\}\subset\partial_\infty\matH^4$, and corresponding to the ideal vertices of the equatorial faces. These are of type EEPPNN, with link a Euclidean rectangular parallelepiped, depicted in Figure \ref{fig:links}-(1);
\item 8 ideal vertices, of which four lie in $\partial_\infty H_{e_4}$ (resp. $\partial_\infty H_{-e_4}$) called \emph{upper} (resp. \emph{lower}) \emph{ideal vertices} corresponding to the ideal vertices of the tetrahedral facets. These are of type TNPPP, with link a Euclidean right prism over an equilateral
triangle, depicted in Figure \ref{fig:links}-(2);
\item 24 finite vertices. These are of type PPNE, with link the spherical tetrahedron depicted in Figure \ref{fig:links}-(3).
\end{enumerate}

\paragraph{Volume.}
As proved in \cite[Proposition 3.21]{riomar}, the volume of the polytope $\calP$ is
\begin{equation}\label{eq:volP} \mathrm{Vol}(\calP)=\frac{4\pi^2}{3}=V_{\mathrm{min}},\end{equation}
which coincides with the minimal volume of a hyperbolic $4$-manifold.

\paragraph{Symmetries.}
Next, we explicitly describe the symmetry group $\mathrm{Sym}(\calP)$ of the polytope $\calP$, its subgroup of orientation-preserving symmetries $\mathrm{Sym}^+(\calP)$, and their action on $\calP$.
An important symmetry of $\calP$ is the \emph{antipodal map}
$$a\colon(x_0,x_1,x_2,x_3,x_4)\mapsto(x_0,-x_1,-x_2,-x_3,-x_4).$$
It is easy to check that $a$ is orientation-preserving and exchanges the two half-spaces $H_{e_4}$ and $H_{-e_4}$. In particular, $a$ exchanges the two tetrahedral facets of $\calP$.

\begin{prop}\label{prop:symmetries}
There is a group isomorphism
$$\mathrm{Sym}(\calP)\cong\matZ/_{2\matZ}\times\mathfrak S_4$$
that restricts to an isomorphism
$$\mathrm{Sym}^+(\calP)\cong\matZ/_{2\matZ}\times\mathfrak A_4,$$
where $\mathfrak S_4$ (resp. $\mathfrak A_4$) is the symmetric (resp. alternating) group on the set of the upper positive facets of the polytope $\calP$.
The center $\matZ/_{2\matZ}$ is generated by the antipodal map $a$.
\end{prop}
\begin{proof}
The group $\mathrm{Sym}(\calP)$ is explicitly computed in \cite[Section 3.2]{riomar} and \cite[Section 4]{KS} as
$$\mathrm{Sym}(\calP)=\langle r,l,m,n\rangle,$$
where
$$r\colon(x_0,x_1,x_2,x_3,x_4)\mapsto(x_0,x_1,x_2,-x_3,-x_4),$$
$$l\colon(x_0,x_1,x_2,x_3,x_4)\mapsto(x_0,x_2,x_1,x_3,x_4),$$
$$m\colon(x_0,x_1,x_2,x_3,x_4)\mapsto(x_0,x_1,x_3,x_2,x_4),$$
$$n\colon(x_0,x_1,x_2,x_3,x_4)\mapsto(x_0,x_1,-x_3,-x_2,x_4).$$

It is easy to check that the group $\langle l,m,n\rangle$ consists precisely of the symmetries of $\calP$ that preserve the half-space $H_{e_4}$.
Moreover, it acts faithfully on the set of the upper positive facets of $\calP$ as its full permutation group (see also the proof of \cite[Lemma 4.15]{riomar}), and thus we have a natural isomorphism $$\langle l,m,n\rangle\cong\mathfrak S_4.$$
Since $r=a\circ m\circ l\circ m\circ n\circ l\circ m$, we also have
$$\mathrm{Sym}(\calP)=\langle a,l,m,n\rangle.$$

Now, the antipodal map $a$ has order two and is in the center of $\mathrm{Sym}(\calP)$, while $\mathfrak S_4$ is centerless.
Thus, the short exact sequence
$$1\to\mathfrak S_4\to\mathrm{Sym}(\calP)\to\matZ/_{2\matZ}\to1$$
(where the third map sends an $s\in\mathrm{Sym}(\calP)$ to $0$ if $s$ preserves $H_{e_4}$, and to $1$ otherwise) splits, and furnishes an isomorphism $\mathrm{Sym}(\calP)\cong\matZ/_{2\matZ}\times\mathfrak S_4$.

Finally, since the upper tetrahedral facet $\calT$ is a regular tetrahedron whose set of faces is
$$\{\calT\cap\calX\ |\ \calX\mathrm{\ an\ upper\ positive\ facet\ of\ }\calP\},$$
the group $\langle l,m,n\rangle$ acts on $\calT$ as its symmetry group.
The subgroup $\mathfrak A_4$ acts by orientation-preserving isometries of $\calT$ and $a$ is orientation-preserving. Thus, the subgroup $\matZ/_{2\matZ}\times\mathfrak A_4$ acts on $\calP$ as its group of orientation-preserving isometries, and the proof is completed.
\end{proof}

From now on, we will naturally write the elements of $\mathrm{Sym}(\calP)$ as elements of $\matZ/_{2\matZ}\times\mathfrak S_4$.

The quotient $\calP/_{\mathfrak S_4}$ is isometric the polytope represented by the Coxeter diagram in Figure \ref{fig:diagrammi}-right -- see \cite[Figure 33]{KS} and \cite[Figure 7]{riomar}.

\subsection{Commensurability}\label{sec:commensurabilty-invariants}
A (complete) \emph{hyperbolic orbifold} is a quotient $\matH^n/_\Gamma$, for a discrete group $\Gamma<\mathrm{Isom}(\matH^n)$.
Two orbifolds $\matH^n/_\Gamma$ and $\matH^n/_{\Gamma'}$ are \emph{commensurable} if $\Gamma\cap g\Gamma'g^{-1}$ has finite index in both $\Gamma$ and $g\Gamma'g^{-1}$, for some $g\in\mathrm{Isom}(\matH^n)$.

A convex polytope $P\subset\matH^n$ is called a \emph{Coxeter polytope} if all its dihedral angles are integral submultiples of $\pi$.
In this case, we interpret $P$ as a hyperbolic orbifold $\matH^n/_\Gamma$, where $\Gamma$ is the hyperbolic Coxeter group generated by reflections through the supporing hyperplanes of $P$.

We refer the reader to \cite{Margulis} for the notion of arithmetic lattices.
We will not describe the regular ideal $24$-cell $\calC$, as it is not necessary for our purposes. The interested reader can find more reference elsewhere in the literature, for instance in \cite{Sasha}.

\begin{prop}\label{prop:PRC-not-commensurable}
The rectified $5$-cell $\calR$, the polytope $\calP$ and the ideal regular $24$-cell $\calC$ are pairwise non-commensurable arithmetic Coxeter polytopes.
\end{prop}

\begin{proof}
The arithmeticity of the $24$-cell is proved in \cite[Section 4]{RT}, while arithmeticity of the polytope $\calP$ is observed in \cite{KS} (it can be easily verified from the Coxeter diagram in Figure \ref{fig:diagrammi}-right by applying Vinberg's algorithm \cite{Vin} as explained in \cite[Section 13.3]{KS}).
The rectified $5$-cell $\calR$ is clearly commensurable with the pyramid $\calR/_{\mathfrak{S}_5}$ of Lemma \ref{lemma:pyramid}, which is shown to be arithmetic in \cite{GJK}.

By the work of Maclachlan \cite{Mac}, commensurability classes of arithmetic Coxeter polytopes are distinguished by the ramification sets of some naturally associated quaternion algebras. As shown in \cite[Proposition 4.25]{riomar}, the ramification set of the $24$-cell is trivial, while the ramification set of the polytope $\calP$ is the set $\{2,5\}$. Finally, the ramification set for the pyramid $\calR/_{\mathfrak{S}_5}$ is computed in \cite[Table 4, first line]{GJK}, and is given by the set $\{3,\infty\}$.

This proves that these three Coxeter polytopes are pairwise non-commensurable.
\end{proof}

\begin{rem}
Proposition \ref{prop:PRC-not-commensurable} allows us to correct the inexact claim in the proof of \cite[Proposition 4.4]{KS2014} that the rectified $5$-cell $\calR$ and the $24$-cell $\calC$ are commensurable. However, since $\calR$ is indeed arithmetic, the statement of that Proposition remains true.
\end{rem}

We are finally ready to prove Theorem \ref{cor:cor1}. The proof follows essentially by combining Proposition \ref{prop:PRC-not-commensurable} with previously known results, and the details are given below.

\begin{proof}[Proof of Theorem \ref{cor:cor1}]
All the manifolds in the Ratcliffe-Tschantz census \cite{RT} are by construction commensurable with the $24$-cell $\calC$ and of volume $\mathrm{Vol}(\calC)=V_{\mathrm{min}}$.

In order to build a minimal-volume manifold commensurable with the rectified $5$-cell $\calR$, consider the manifold $\calZ$ described in \cite[Remark 4.4]{S}.
Since this manifold is tessellated by six copies of the rectified $5$-cell $\calR$, we have $\mathrm{Vol}(\calZ)=6\cdot \mathrm{Vol}(\calR)=V_{\mathrm{min}}$. Now, the manifold $\calZ$ has totally geodesic boundary isometric to the complement of the figure-eight knot. It is sufficient to ``kill'' the boundary component by taking its quotient under the map which produces the Gieseking manifold, to obtain an orientable, minimal-volume hyperbolic $4$-manifold $\cal{Z}'$ with empty boundary. As can be seen immediately from the construction, the manifold $\calZ'$ is built by gluing copies of $\calR$ via symmetries of $\calR$, so Lemma \ref{lemma:commensurability} applies showing that $\cal{Z}'$ is in fact commensurable with $\calR$. 

By Proposition \ref{prop:PRC-not-commensurable}, the polytopes $\calR$ and $\calC$ are non-commensurable, and we conclude the same for the manifold $\cal{Z}'$ and any manifold from the Ratcliffe-Tschantz census. 
\end{proof}

\section{A manifold with twice the minimal volume}\label{sec:minvolume}

In this Section, we first prove the following:
\begin{teo}\label{teo:main1}
The commensurability class of the Kerckhoff-Storm polytope $\calP$ contains a non-orientable manifold $\calN$ with $\chi(\calN)=2$.  
\end{teo}

In order to build such a manifold $\calN$, we will first construct a hyperbolic $4$-manifold $\calM$ with totally geodesic boundary and such that $\chi(\calM)=1$, by gluing in pairs through isometries the facets of the Kerckhoff-Storm polytope $\calP$ introduced in Section \ref{sec:polytope}. The manifold $\calN$ will then be obtained by mirroring the manifold $\calM$ in its boundary. 

Finally, in Section \ref{sec:nonarithemtic}, we will prove Theorem \ref{teo:main3} by ``interbreeding''.


\subsection{Defining the manifold $\calM$}

Before defining the manifold $\calM$, we first define the maps to be used as face-pairings of the polytope $\calP$.

\paragraph{The figure-eight knot pattern.}
Similarly to \cite[Section 4.3]{riomar}, in order to glue some facets of $\calP$ we will exploit the usual ideal triangulation of the figure-eight knot complement made of two regular tetrahedra, as shown in Figure \ref{fig:triangulation_figure8}.
The resulting complex contains four triangular faces $P,J,F,R$ and two edges.
Each edge of the complex has valence six and the return map around an edge is trivial.
We call $\phi_P,\phi_J,\phi_F$ and $\phi_R$ the corresponding face pairings.

\begin{figure}[htbp]
\labellist
\hair 10pt
\pinlabel {$\calT$} at -20 60
\pinlabel {$a(\calT)$} at 370 60
\endlabellist
\centering
\includegraphics[width=9 cm]{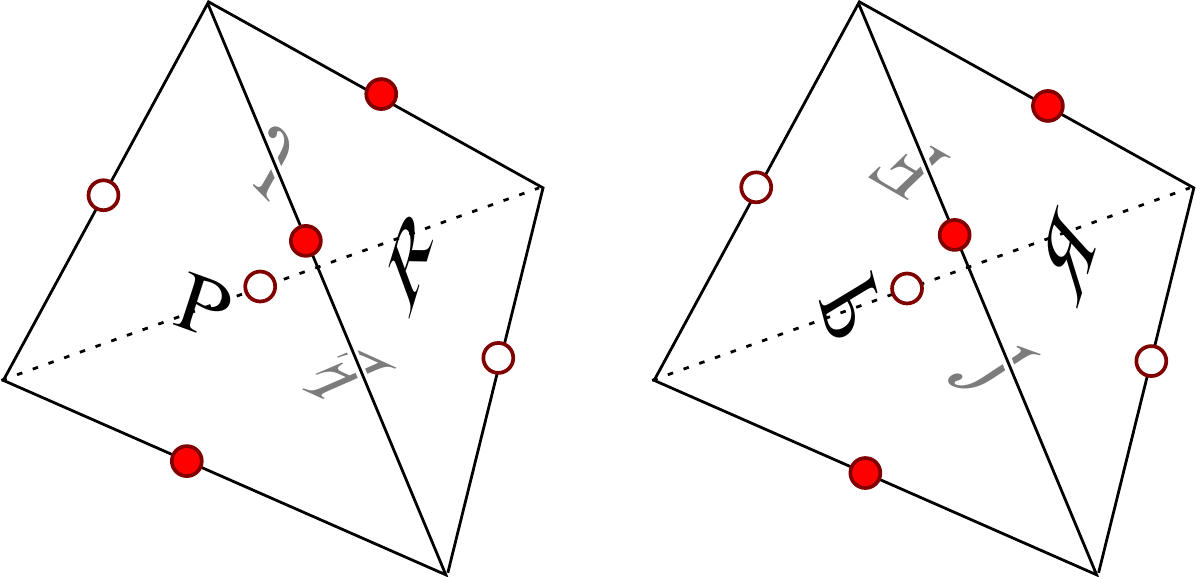}
\nota{This is the well-known ideal triangulation of the figure-eight knot complement.
The two edges of the triangulation have valence six and are marked with
solid and hollow dots.
The resulting complex contains four triangular faces, labeled $P$, $J$, $F$ and $R$.
The orientation of the labels in the figure determines the gluing pattern.
We identify the upper tetrahedral facet $\calT$ of $\calP$ with the tetrahedron on the left and the lower tetrahedral facet $a(\calT)$ with the tetrahedron on the right, in such a way that for every upper positive facet $\calF$ of $\calP$, the face $\calT\cap\calF$ has the same label $X\in\{P,J,F,R\}$ as its antipode $a(\calT\cap\calF)$.}\label{fig:triangulation_figure8}
\end{figure}

By calling $P',J',F',R'$ and $P'',J'',F'',R''$ respectively the faces of the tetrahedron on the left in Figure \ref{fig:triangulation_figure8} and of that on the right in the obvious way, the face pairings are induced by the maps
$$\phi_P:(P',J',R',F')\rightarrow (P'',F'',J'',R''),$$
$$\phi_R:(P',J',R',F')\rightarrow (F'',P'',R'',J''),$$
$$\phi_F:(P',J',R',F')\rightarrow (R'',P'',J'',F''),$$
$$\phi_J:(P',J',R',F')\rightarrow (F'',J'',P'',R''),$$
so that $X'$ is glued with $X''$ through the map $\phi_X$ for each $X\in\{P,R,F,J\}$.

The two tetrahedra of the triangulation may be interpreted as the two tetrahedral facets of the polytope $\calP$. Each of them is adjacent along its triangular faces to four positive facets of $\calP$.

At once, we fix an identification of the upper tetrahedral facet $\calT$ of $\calP$ with the tetrahedron on the left and the lower tetrahedral facet $a(\calT)$ with the tetrahedron on the right, in such a way that for every upper positive facet $\calF$ of $\calP$, the face $\calT\cap\calF$ has the same label $X\in\{P,R,F,J\}$ as its antipode $a(\calT\cap\calF)$.

\begin{rem}\label{rem:PRFJ}
There is a one-to-one correspondence between the 2-strata $\{P,R,F,J\}$ of the figure-eight knot complement's triangulation and the upper (or lower) positive facets of $\calP$.
In particular, recalling Proposition \ref{prop:symmetries}, the $\mathfrak S_4$-factor of $\mathrm{Sym}(\calP)$ is identified with the permutation group of the set $\{P,R,F,J\}$.
\end{rem}

\begin{rem}\label{rem:equatorial_pairs}
Moreover, every pair of positive (or negative) facets of $\calP$ is adjacent to exactly one equatorial facet (this can be seen in Figure \ref{fig:facets}). Thus, there is a one-to-one correspondence between the edges of a single chosen tetrahedron (for instance the one on the left) and the equatorial facets of $\calP$.
\end{rem}

\paragraph{Gluing the positive facets.}
Let us begin by defining the face pairings on the positive facets of $\calP$.
We wish to do this in such a way that the tetrahedral facets are glued together as in the ideal triangulation of the figure-eight knot complement of Figure \ref{fig:triangulation_figure8}.

Note that each of the face pairings of that triangulation induces a unique simplicial map between the two tetrahedra, and therefore it defines a bijection from the set of the upper positive facets of $\calP$ to the set of the lower positive facets.
In particular, by Proposition \ref{prop:symmetries}, for each of these pairing maps there is a unique symmetry of the polytope $\calP$ which acts on the positive facets in the prescribed way. By a slight abuse of notation, we call these symmetries $\phi_P, \phi_R, \phi_F, \phi_J\in\mathrm{Sym}(\calP)$.
Recalling Remark \ref{rem:PRFJ}, these maps can be described as follows:
\begin{equation}\label{eq:positive_maps}
\begin{matrix}
\phi_P=a\circ(JFR),&&
\phi_R=a\circ(PFJ),\\
\phi_F=a\circ(PRJ),&&
\phi_J=a\circ(PFR).
\end{matrix}
\end{equation}

\paragraph{Gluing the tetrahedral facets.}
Let us now define the face pairings on the tetrahedral facets of the polytope $\calP$. In this case, we will always use restrictions of the same symmetry of $\calP$.

There exists an orientation-reversing, fixed-point-free isometric involution $g$ of the figure-eight knot complement.
The quotient of the figure-eight knot complement under $g$ is the Gieseking manifold, which is the cusped hyperbolic $3$-manifold of minimal volume.

In terms of the action on the triangulation in Figure \ref{fig:triangulation_figure8}, the involution $g$ exchanges the two tetrahedra, and acts on the triangular faces of the complex as the permutation
$$P\leftrightarrow F,\; R\leftrightarrow J.$$

Again, by Proposition \ref{prop:symmetries}, there is a unique isometric involution $g$ of the polytope $\calP$ which maps the upper positive facets to the lower positive facets in the prescribed way.
Recalling Remark \ref{rem:PRFJ}, this map is described by
\begin{equation}\label{eq:giseking-positive}
g=a\circ(PF)(JR)\in\mathrm{Sym}(\calP).
\end{equation}

An important consequence of the fact that the map $g$ induces an automorphism of the triangulation of the figure-eight knot complement is that it preserves the face pairings on the positive facets of the polytope $\calP$. This is expressed by the equations
\begin{equation}\label{eq:positive-letters}
 g\circ \phi_P\circ g= \phi_F^{-1},\quad g\circ \phi_R\circ g= \phi_J^{-1}.
\end{equation}

\paragraph{Gluing the negative facets.}
Finally, we define the face pairings for the negative facets of $\calP$. Once more, we wish to choose restrictions of a symmetry of $\calP$ which induces an isometry of the figure-eight knot complement. Note that there is an isometric involution of the figure-eight knot complement which maps each tetrahedron to itself and acts on the triangular faces through the permutation
$$P\leftrightarrow R,\; F\leftrightarrow J.$$
Again, by Proposition \ref{prop:symmetries}, this permutation of the upper positive facets induces an isometric involution $i$ of the polytope $\calP$. Recalling Remark \ref{rem:PRFJ}, its description is 
\begin{equation}\label{eq:involution-positive}
i=(PR)(FJ)\in\mathrm{Sym}(\calP).
\end{equation}
Note that, in contrast with the previously chosen pairing maps, the involution $i$ lies in the $\mathfrak S_4$-factor of $\mathrm{Sym}(\calP)$.

Once again, let us note that the fact that $i$ induces an automorphism of the figure-eight knot complement is expressed by the equations
\begin{equation}\label{eq:positive-negative}
i\circ \phi_P \circ i=\phi_R,\quad i\circ \phi_F \circ i = \phi_J.
\end{equation}

We are finally ready to define the desired manifold $\calM$. 
\begin{defn}\label{def:manifold}
We define $\calM$ to be the space obtained from the polytope $\calP$ as follows:
\begin{itemize}
\item each point $p$ of an upper positive facet $X$ is identified with $\phi_X(p)$, where $\phi_X$ is defined by \eqref{eq:positive_maps}; 
\item each point $p$ of a tetrahedral facet is identified with $g(p)$, where $g$ is defined by \eqref{eq:giseking-positive};
\item each point $p$ of a negative facet is identified with $i(p)$, where $i$ is defined by
\eqref{eq:involution-positive}.
\end{itemize}
\end{defn}

\subsection{Proof of Theorem \ref{teo:main1}}\label{sec:hyperbolicstructure}

The faces of the polytope $\calP$ induce a natural stratification of the complex $\calM$.
In the sequel, we prove that $\calM$ is indeed a complete hyperbolic manifold with totally geodesic boundary. Note that, in order to define the manifold $\calM$, we have paired all the facets of the polytope $\calP$ except the equatorial ones. The equatorial facets will tessellate the boundary of $\calM$. Let us first check that the chosen pairing maps do not introduce self-pairings between neither positive, negative or tetrahedral facets of $\calP$. This is fairly obvious in the case of positive facets, since the upper ones are paired to the lower ones, and in the case of the two tetrahedral facets (which are paired to each other).

Note that every upper (resp. lower) positive facet is adjacent to a \emph{unique} lower (resp. upper) negative facet. This induces a natural one-to-one correspondence between the set of positive facets and the set of negative facets of $\calP$. In terms of the corresponding space-like vectors, it is given by $f\left(x_0,x_1,x_2,x_3,x_4\right)=\left(x_0,x_1,x_2,x_3,-1/{x_4}\right)$.  Now, by \eqref{eq:involution-positive}, the involution $i$ does not preserve any of the positive facets; therefore, it does not preserve any of the negative ones.

The associated correspondence between the positive and negative facets of $\calP$ sends each upper (resp. lower) positive facet to the unique lower (resp. upper) negative facet adjacent to it.
The function $f$ does not induce an isometry of $\matH^4$, but does nonetheless commute with all elements of $\mathrm{Sym}(\calP)$. Therefore,
the action of any element of $\mathrm{Sym}(\calP)$ on the set of negative facets can be directly inferred by its action on the set of positive facets, and by composing it with the function $f$.
Now, by \eqref{eq:involution-positive}, the involution $i$ does not preserve any of the positive facets; therefore, it does not preserve any of the negative ones. 

Following Thurston \cite{notes} (see also Ratcliffe \cite[Chapter 11]{ratcliffe}), we can reduce the issue of proving that $\cal{M}$ is a hyperbolic manifold with totally geodesic boundary to a purely $3$-dimensional problem: it suffices to check that the links of the ideal vertices of $\calP$ are paired together to produce Euclidean $3$-manifolds (perhaps with totally geodesic boundary), and that the links of the finite vertices (which are all adjacent to some equatorial facet) are paired together to produce $3$-dimensional hemispheres.

\paragraph{The equatorial ideal vertices.}
Let us begin with the equatorial ideal vertices of $\calP$. Recall that the link of each such vertex is a rectangular parallelepiped, shown in Figure \ref{fig:links}-(1). By gluing these $12$ parallelepipeds together according to the pairing maps, we obtain a piecewise Euclidean complex $\calE$. We now show that $\calE$ is indeed a (possibly disconnected) Euclidean $3$-manifold with totally geodesic boundary.

Now, the various edges of the parallelepiped fall into three types, according to the type of the ridges of $\calP$ which they correspond to: PN, PE and NE.
Note, moreover, that all the pairing maps preserve the facet type. Therefore, also the edges of the complex $\calE$ fall into the types PN, PE and NE. The edges of the first type lie in the interior of the complex $\calE$, while the edges of types PE and NE lie in its boundary.

Consider the abstract graph $\calG$ whose nodes correspond to the edges of type PN of the above parallelepipeds, with an arc connecting the nodes corresponding to the edges $E_1$ and $E_2$ if there is a pairing map between the faces of the corresponding two parallelepipeds mapping $E_1$ to $E_2$. Since an edge of type PN is adjacent to exactly two paired rectangular faces, the graph $\calG$ is a union of cycles, called the \emph{edge cycles} of the gluing. To each edge cycle, there corresponds a \emph{return map} from any of the edges which make up the cycle to itself: simply follow the sequence of pairing maps until the cycle closes up. There are clearly only two possibilities for each return map: either it is the identity or it acts by exchanging the vertices.

Since all the parallelepipeds are right-angled, in order to check that the edges of type PN of the complex $\calE$ are non-singular, we need to check that all these edge cycles have length $4$ and that all the return maps are trivial. The latter condition assures that the links of the midpoints of the edges of the complex $\calE$ are indeed spheres and not projective planes.

The pairing maps involved for each edge of type PN fall into two types: the involution $i$ for the negative facets and the pairing maps $\phi_P$, $\phi_J$, $\phi_F$ and $\phi_R$ for the positive ones. These two types of pairing maps clearly alternate in each edge cycle. Now, because of Equation \eqref{eq:positive-negative}, we observe that the sequences of pairing maps are of one of the two types
\begin{equation}\label{eq:face-cycles}i\circ \phi_P \circ i\circ {\phi_R}^{-1}=\mathrm{id},\quad i\circ \phi_F \circ i \circ {\phi_J}^{-1} = \mathrm{id}.\end{equation}
The equation above proves that these face cycles have length at most $4$. In order to verify that their length is exactly $4$, we need to take a closer look at the behavior of the pairing maps on the positive facets, and verify that none of the ridges of $\calP$ of type PN is mapped to itself under $\phi_R^{-1}$, $i \circ \phi_{R}^{-1}$ or $\phi_P\circ i \circ \phi_R^{-1}=i$ (and similarly, that the equivalent statement holds for the cycles of the second type).

Note that each ridge of type PN determines a unique positive facet adjacent to it, and that these fall into upper and lower facets. The isometry $i$ preserves these two groups, while the maps $\phi_P$, $\phi_J$, $\phi_F$ and $\phi_R$ exchange them. This implies that none of these latter maps can take a ridge of type PN to itself, and neither can the maps $i \circ \phi_{R}^{-1}$ or $i \circ {\phi_J}^{-1}$.
Finally, the involution $i$ does not preserve any of the positive facets by Equation \eqref{eq:involution-positive}. Note that by Equation \eqref{eq:face-cycles} all the return maps are induced by the identity map of the polytope $\calP$, and therefore all return maps are trivial. 

Concerning the edges of $\calE$ of type PE and NE, they are also easily seen to be non-singular. This is a consequence of the fact that all the dihedral angles between the equatorial facets of $\calP$ and the positive and negative facets are right.

Having shown that the edges of the complex $\calE$ are non-singular, we now turn our attention to its vertices. Note that all the vertices lie in the boundary of $\calE$. 
Let us call $\calL$ the link of each such vertex, which is tessellated by right-angled spherical triangles.  In order to prove that the complex $\calE$ is a Euclidean manifold, it remains to show that $\calL$ is isometric to a hemisphere of $\mathbb{S}^2$.  By the previous argument concerning the edges of type PN, we obtain that $\calL$ is indeed a (non-singular) spherical surface with totally geodesic boundary: therefore it can only be a hemisphere.   

\paragraph{The upper and lower ideal vertices.}
We now deal with the upper and lower ideal vertices of $\calP$. Recall that the link of each such vertex is a right prism over an equilateral triangle, shown in Figure \ref{fig:links}-(2).

The pairings along the positive facets glue these prisms together along their rectangular faces to produce a Euclidean manifold $T\times I$, where $T$ is the torus corresponding to the cusp section of the figure-eight knot complement and $I$ is a closed interval. The torus $T$ is tessellated by eight Euclidean equilateral triangles, one for each upper or lower ideal vertex of $\calP$. The eight triangles which tessellate the boundary component $T\times\{0\}$ correspond to the vertex figures of the two tetrahedral facets, while the triangles which tessellate $T\times \{1\}$ correspond to the ideal vertex figures of negative facets.

Therefore, the pairing maps between the tetrahedral facets induce an isometric involution of the torus $T\times\{0\}$, and in  a similar way the pairings between the negative facets induce an isometric involution of the torus $T\times\{1\}$. It is sufficient to check that both involutions are fixed-point free.

The pairings along the tetrahedral facets come from the involution $g$ defined by \eqref{eq:giseking-positive}. The involution $g$ defines a fixed-point-free involution of the figure-eight knot complement such that the quotient is the Gieseking manifold. Therefore, also its action on the torus $T\times\{0\}$ is fixed-point-free, and the quotient is a Klein bottle tessellated by four equilateral triangles.

In the case of the negative facets, the pairing map is given by the involution $i$ defined by \eqref{eq:involution-positive}. This induces an orientation-preserving involution of the figure-eight knot complement which is, however, not fixed-point-free. The set of fixed points is a knot $K$ in the figure-eight knot complement. From Figure \ref{fig:triangulation_figure8}, one observes that the intersection of the knot $K$ with each of the two tetrahedra tessellating the figure-eight knot complement is a geodesic segment connecting the midpoints of a pair of opposite edges: in particular it is disjoint from small enough horospheres around the ideal vertices. This implies that the action of the involution $i$ on the torus $T\times\{1\}$ is indeed fixed-point-free. The quotient is a Euclidean torus tessellated by four equilateral triangles.

Summarizing the above discussion, the effect of the chosen pairing maps is to identify the faces of the links of the upper and lower ideal vertices of $\calP$ to produce a Euclidean manifold with a singular fibration over the interval $I$. The fiber above $0$ is a one-sided Klein bottle, while the fiber above $1$ is a one-sided torus. All other fibers are two-sided tori, tessellated by $8$ equilateral triangles.

\paragraph{The finite vertices.}
We lastly deal with the finite vertices of $\calP$. Since all the finite vertices are adjacent to some equatorial facet, these correspond to points in the boundary of $\calM$. Recall that the link of each such vertex is the spherical tetrahedron shown in Figure \ref{fig:links}-(3), which has one edge with dihedral angle $\pi/3$ and all the others right-angled. 

These 24 spherical tetrahedra are glued together via the pairing maps to produce a piecewise spherical complex $\calS$, corresponding to the link of a vertex lying in the boundary of $\calM$. We have to show that each component of $\calS$ is indeed homeomorphic to a $3$-dimensional disk, realized as a hemisphere of $\matS^3$.

We begin with some preliminary considerations that will be useful later on.
We have already proven that the cusp sections of $\calM$ are Euclidean manifolds. Therefore, all the interior points of the unbounded strata of $\calM$ are non-singular.
Since every ridge of the polytope $\calP$ has at least one ideal vertex, we know that the 2-strata of $\calM$ are non-singular.
This translates into the fact that every edge of the complex $\calS$ is non-singular.

For the same reason, all the edges of the complex $\calM$ which connect an ideal vertex to a finite vertex are non-singular.
This translates to the fact that every vertex of $\calS$ of type PPN and PNE is also non-singular.

Let us now come back to the 24 spherical tetrahedra that tessellate the complex $\calS$.
By gluing them together via the pairing maps of positive facets, we obtain four spherical polyhedra, each isometric to the intersection of two orthogonal half-spaces of $\matS^3$. 
In this regard, recall that the figure-eight knot's ideal triangulation has all edges of valence $6$ and note that the types of all the strata are preserved by isometries of $\calP$. The result is represented in Figure \ref{fig:arancia}.

\begin{figure}[htbp]
\begin{center}
\includegraphics[width=0.3\textwidth]{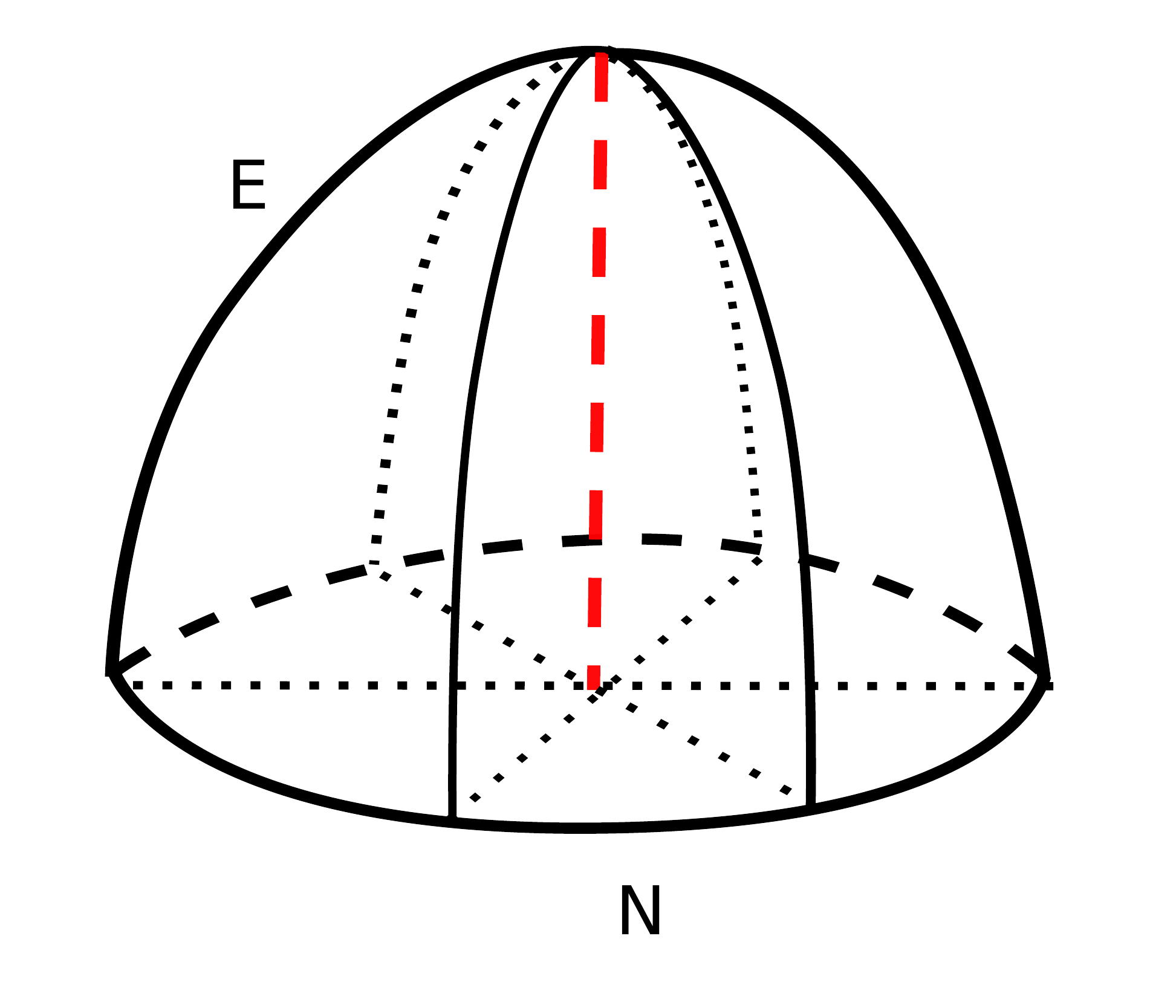}
\nota{The result of the gluing of the spherical vertex figure along the positive facets is given by four copies of the spherical polyhedron above. Each copy is tessellated by six spherical tetrahedra. The two faces of this polyhedron correspond respectively to equatorial (on the top) and  negative (on the bottom) facets of $\calP$.}\label{fig:arancia}
\end{center}
\end{figure}

The boundary of this spherical polyhedron is obviously a union of two disks, the first one is tessellated by triangles corresponding to the negative facets of $\calP$, while the second one is tessellated by triangles corresponding to the equatorial facets.

Let us now glue the faces of the resulting polyhedron via the pairing maps corresponding to the negative facets. These are induced by the involution $i$. We claim that the four disks corresponding to the negative facets are identified in pairs by isometries. The induced pairing maps are isometries on each triangle and extend continuously with their inverses to each disk by \eqref{eq:positive-negative}. There remains to show that no disk is mapped to itself under the pairing maps. However, if this were the case, the center of such disk would be a fixed point, and this would imply the existence of a singular point corresponding to a vertex of type PPN, which was previously excluded.

This proves that the gluings at the spherical vertex links along the positive \emph{and} negative facets produces two balls $B_1$ and $B_2$, each isometric to a hemisphere of $\mathbb{S}^3$ and tessellated by $12$ spherical tetrahedra. 

\paragraph{Conclusion of the proof.}
We have finally shown that $\calM$ is a hyperbolic 4-manifold with totally geodesic boundary. Since the manifold $\calM$ is tessellated by a single copy of the polytope $\calP$, by \eqref{eq:volP} we have
$$\mathrm{Vol}(\calM)=\frac{4\pi^2}{3}=V_{\mathrm{min}}.$$

\begin{rem}
The manifold $\calM$ is non-orientable. The reason for this is that the symmetries of $\calP$ used to define the pairing maps on the facets are orientation-preserving.
\end{rem}

To conclude the proof of Theorem \ref{teo:main1}, we simply define
$$\calN=\calD(\calM)$$
to be the \emph{double} of $\calM$, that is, the hyperbolic manifold obtained from two copies of $\calM$ by identifying their boundary through the map induced by the identity.
Clearly, the volume of $\calN$ is equal to $8\pi^2/3=2\cdot V_{\mathrm{min}}$.

Finally, we show that the manifold $\calM$ is commensurable with the polytope $\calP$. This is a straightforward consequence of the following:

\begin{lemma}\label{lemma:commensurability}
Let $M$ be a hyperbolic manifold obtained by pairing the facets of some copies of a Coxeter polytope $P$. Suppose that each pairing map is induced by a symmetry of the polytope $P$. Then, the orbifolds $M$ and $P$ are commensurable.

\end{lemma}
\begin{proof}
By analyzing a holonomy representation for the hyperbolic structure of $M$, it is not difficult to conclude that $M$ covers the orbifold $P/_{\mathrm{Sym}(P)}$.
\end{proof}

The proof of Theorem \ref{teo:main1} is completed.

\section{Non-arithmetic manifolds}\label{sec:nonarithemtic}

In this section, we prove Theorem \ref{teo:main3} by the ``interbreeding'' technique.
Let us first recall Gromov and Piatetski-Shapiro's theorem.

\begin{teo}[Gromov, Piatetski-Shapiro \cite{G-PS}]\label{teo:G-PS}
Let $M_1$ and $M_2$ be complete, finite volume, hyperbolic manifolds with non-empty, totally geodesic boundary.
Suppose that $\partial M_1$ and $\partial M_2$ are isometric, and let $\phi:\partial M_1\rightarrow \partial M_2$ be an isometry.
Let $M$ be the hyperbolic manifold obtained by gluing $M_1$ and $M_2$ through the isometry $\phi$.
If $M$ is arithmetic, then it is commensurable with the doubles $\calD(M_1)$ and $\calD(M_2)$.
\end{teo}

A straightforward consequence of this theorem is that if $\calD(M_1)$ and $\calD(M_2)$ are incommensurable, then the manifold $M$ constructed by gluing $M_1$ to $M_2$ along their totally geodesic boundaries is non-arithmetic.

\begin{proof}[Proof of Theorem \ref{teo:main3}]
First of all, we observe that the manifold $\calN$ of Theorem \ref{teo:main1} contains two totally geodesic copies $\calG_1$ and $\calG_2$ of the Gieseking manifold, which are the result of the gluing of the tetrahedral facets of $\calP$.

We call $\calN_{\slash\!\!\slash}$ the manifold obtained by cutting $\calN$ along the hypersurface $\calG_1$. Clearly, $\calN_{\slash\!\!\slash}$ can be obtained from two copies of the polytope $\calP$ but without pairing the tetrahedral facets of one of them.
In particular, we get that $\partial\calN_{\slash\!\!\slash}$ is isometric to the figure-eight knot complement.

Recall that $\calN$ (and thus $\calN_{\slash\!\!\slash}$) is non-orientable. Let $\widetilde{\calN_{\slash\!\!\slash}}$ be the orientable double covering of $\calN_{\slash\!\!\slash}$.
The manifold $\widetilde{\calN_{\slash\!\!\slash}}$ has two boundary components, each isometric to the figure-eight knot complement.
As we already noted, the latter 3-manifold has an isometric fixed-point-free involution $\iota$.
We call $\calN'$ the hyperbolic manifold obtained by quotienting one of the two components of $\partial\widetilde{\calN_{\slash\!\!\slash}}$ by $\iota$.

Let us now consider the orientable hyperbolic $4$-manifold $\calZ$ described in \cite[Remark 4.4]{S}. Its boundary is totally geodesic and isometric to the figure-eight knot complement.
To build the desired manifold $\calH$, we simply glue $\calN_{\slash\!\!\slash}$ to $\calZ$ through any isometry of their boundaries.
Similarly, to build the manifold $\calH'$, we glue $\calN'$ to $\calZ$.

By construction (recall Lemma \ref{lemma:commensurability}), the manifold $\calD(\calZ)$ is commensurable with the rectified 5-cell $\calR$, while $\calD(\calN_{\slash\!\!\slash})$ and $\calD(\calN')$ are commensurable with the polytope $\calP$.
The non-arithmeticity of $\calH$ and $\calH'$ follows from Proposition \ref{prop:PRC-not-commensurable} and Theorem \ref{teo:G-PS}.
\end{proof}

\begin{table}[!hb]
\centering

\begin{tabular}{lll}
 \begin{tabular}[c]{@{}l@{}} \it Stefano Riolo\\ \it Institut de math\'ematiques\\ \it University of Neuch\^atel\\ \it Rue Emile-Argand 11\\ \it CH-2000 Neuch\^atel\\ \it Switzerland\\ \it stefano (dot) riolo (at) unine.ch\end{tabular} & \begin{tabular}[c]{@{}l@{}} \it Leone Slavich\\ \it Department of Mathematics\\ \it University of Pisa\\ \it Largo Pontecorvo 5\\ \it I-56127 Pisa\\ \it Italy\\ \it leone (dot) slavich (at) gmail.com\end{tabular}
\end{tabular}
\end{table}

\end{document}